\documentclass[12pt]{amsart}
\usepackage{amssymb,epsfig,amsmath,latexsym,amsthm}
\usepackage{graphicx}

\usepackage{float}
\usepackage[all,cmtip]{xy}

\usepackage{enumitem}
\usepackage{comment}
\usepackage[hidelinks]{hyperref}
\theoremstyle{plain}
\newtheorem{theorem}{Theorem}[section]
\newtheorem{lemma}[theorem]{Lemma}
\newtheorem{claim}{Claim}
\newtheorem{proposition}[theorem]{Proposition}
\newtheorem{corollary}[theorem]{Corollary}

\theoremstyle{definition}

\newtheorem{remark}[theorem]{Remark}

\newtheorem{example}[theorem]{Example}

\usepackage{epsfig,color}

\makeatother

\theoremstyle{definition}

\def\fnum{equation} 

\newtheorem{Cor}[\fnum]{Corollary}

\newtheorem{Claim}{Claim}[subsection]

\numberwithin{equation}{section}

\usepackage[hidelinks]{hyperref}

\begin{document}
\title
[Obstructions to nonnegative scalar and mean curvatures]
{Topological obstructions to nonnegative scalar curvature and mean convex boundary}

\author{Ezequiel Barbosa}
\address{Universidade Federal de Minas Gerais (UFMG), Caixa Postal 702, 30123-970, Belo Horizonte, MG, Brazil}
\email{ezequiel@mat.ufmg.br}
\author[Franciele Conrado]{Franciele Conrado$^{2}$} \address{$^{2}$Instituto de Ci\^{e}ncias Exatas-Universidade Federal de Minas Gerais\\ 30161-970-Belo Horizonte-MG-BR} \email{conrado@mat.ufmg.br} 
\thanks{The authors were partially supported by CNPq/Brazil and Capes/Brazil agencies.}

\begin{abstract}
We study topological obstructions to the existence of a Riemannian metric on manifolds with boundary such that the scalar curvature is non-negative and the boundary is mean convex. We construct many compact manifolds with boundary which admit no Riemannian metric with non-negative scalar curvature and mean convex boundary. For example, we show that the manifold  $(T^{n-2}\times \Sigma )\# N$, where $\Sigma$ is a compact, connected and orientable surface which is not a disk or a cylinder and $N$ is a closed $n$-dimensional manifold, does not admit a metric of non-negative scalar curvature and mean convex boundary, and the manifold $(I\times T^{n-1})\#N$, where $I=[a,b]$, does not admit a metric of positive scalar curvature and mean convex boundary.

\end{abstract}

\maketitle

\section{Introduction}\label{intro}

A central problem in modern differential geometry concerns the connection between curvature and topology of a manifold. Especially, if the problem is when a given manifold  admits a Riemannian metric with positive or non-negative scalar curvature. We will not go over the case of closed manifolds, instead, our focus here will be on compact manifolds with non-empty boundary. For the case of closed manifolds, see the important works due to Schoen-Yau \cite{SCHYau2}, \cite{SCHYau}, and Gromov-Lawson \cite{GL1}, \cite{GL2}, \cite{GL3}.

Consider, for instance, the case of surfaces.  Let $(M^2,g)$ be an orientable compact two-dimensional Riemannian manifold with non-empty boundary $\partial M$. The Gauss-Bonnet Theorem states that
\[
\int_{M}Kda+\int_{\partial M}k_gds=2\pi\chi(M)\,,
\]
where $K$ denotes the Gaussian curvature, $k_g$ is the geodesic curvature of the boundary, $\chi(M)$ is the Euler characteristic, $da$ is the element of area and $ds$ is the element of length. Note that the  invariant $\chi(M)$ gives a topological obstruction to the existence of  certain types of Riemannian metrics on the surface $M^2$. For instance, a compact surface $M^2$ with negative (non-positive) Euler characteristic does not admit a Riemannian metric with non-negative (positive) Gaussian curvature and non-negative geodesic curvature.

In higher dimensions, the relationship between curvature and topology  is much more complicated. A classical theorem due to Gromov \cite{Grom2}, for example, states that every compact manifold with non-empty boundary admits a Riemannian metric of positive sectional curvature.

However, there are topological obstructions if one further imposes geometric restrictions on the boundary. For instance, a result of Gromoll \cite{Grom1} states that a compact Riemannian manifold of positive sectional curvature with non-empty convex boundary is diffeomorphic to the standard disc. Observe, however, that these hypothesis are rather strong because they involve the sectional curvature and not the scalar curvature. Recall that, by the Bonnet-Mayers Theorem, a 3-dimensional manifold with positive Ricci curvature and convex boundary (positive definite second fundamental form) is diffeomorphic to a 3-ball.

The problem of determining topological obstructions for the existence of a metric with nonnegative scalar curvature and mean convex boundary is more subtle. For instance, one such obstruction appears when there exists a compact, orientable and essential surface properly embedded in $M$ which is not a disk or a cylinder. This is the case, for example, if we consider the manifold $M=\mathbb{S}^1\times \Sigma$, where $\Sigma$ is a compact and orientable surface which is not a disk or a cylinder. Indeed, the surface $\{1\}\times \Sigma$ is essential in $M$, so this manifold carries no metric with non-negative scalar curvature and mean convex boundary. If $M$ contains an essential cylinder, then there may exists such a metric on $M$. This is the case, for example, of the manifold $M=I\times T^2$, where $T^2$ denote the torus $\mathbb{S}^1\times \mathbb{S}^1$. Such manifold contains an essential cylinder and have a flat Riemannian metric with totally geodesic boundary. 

From now on, we use the notation $(M,\partial M)$ to represent a compact and orientable manifold with non-empty boundary $\partial M$. Moreover, $R_g$ and $H_g$ denote the scalar curvature of $(M,g)$ and the mean curvature of the boundary $\partial M$ with respect to the outward unit normal vector field on the boundary, respectively.

In this paper, our first result gives a topological obstruction for those 3-dimensional compact manifolds which possess a certain type of surfaces as connected components of their boundaries. 

\begin{theorem}\label{main1}
Let $(M,\partial M)$ be a compact 3-dimensional manifold. Assume that the connected components of $\partial M$ are spheres or incompressible tori, but at least one of the components is a torus. Then there is no Riemannian metric on $M$ with positive scalar curvature and mean convex boundary. In particular, if there exists a Riemannian metric $g$ on $M$ with non-negative scalar curvature and mean convex boundary then  $(M,g)$ is flat with totally geodesic boundary.
\end{theorem}

As a consequence of the theorem above, we obtain that the 3-dimensional manifolds $(\mathbb{S}^1\times \mathring{T}^2)\# N$ and $(\mathbb{S}^1\times \mathring{T}^2)\# (I\times \mathbb{S}^2)$ have no metric with non-negative scalar curvature and mean convex boundary, where $\mathring{T}^2$ is a torus minus an open disk, $I=[a,b]$ and $N$ is a closed, connected and orientable 3-dimensional manifold. Moreover, the 3-dimensional manifolds  $(I\times T^2)\# (I\times T^2)$,  $(\mathbb{S}^1\times \mathring{T}^2)\#(\mathbb{S}^1\times \mathring{T}^2)$, $(I\times T^2)\# (S^1\times \mathring{T}^2)$, $(I\times T^2)\# (I\times \mathbb{S}^2)$, $(\mathbb{S}^1\times \mathring{T}^2)\#  (I\times \mathbb{S}^2)$ have no metric with positive scalar curvature and mean convex boundary. Also, let $N$ be a closed 3-dimensional manifold. Then the manifold $(I\times T^2)\# N$ has no metric with positive scalar curvature and mean convex boundary. If it has a metric with non-negative scalar curvature and mean convex boundary, it is flat with totally geodesic boundary. Thus, from this last claim, we can glue two copies of $(I\times T^2)\# N$ along the boundary and build a flat closed 3-dimensional manifold which is a connected sum of a 3-dimensional torus and a closed 3-dimensional manifold. 

With that discussion above, we obtain the following classification result. 

\begin{Cor} Let $(M,\partial M)$ be a smooth 3-dimensional manifold such that $\partial M$ is the disjoint union of exactly one torus and $k$ spheres, $k\geq 0$. If $M$ has a metric with non-negative scalar curvature and mean convex boundary then

$$M=N\#(\mathbb{S}^1\times \mathbb{D}^2)\#^k \mathbb{B}^3,$$
\noindent where $N$ is a closed 3-dimensional manifold.
\end{Cor}

At this point, one should mention two important facts. First, Gromov-Lawson (see Theorem 5.7 in \cite{GL2}), pointed out that if a compact manifold with boundary possesses metrics with positive scalar curvature and strictly mean convex boundary then its double can be endowed with a metric of positive scalar curvature. Therefore, the problem of characterising the compact manifolds with boundary supporting a metric with positive scalar curvature and strictly mean convex boundary reduces to the problem on theirs doubles manifolds. This was made in a very recent work due to A. Carlotto and C. Li \cite{ACLI}. Second, despite our results are not a complete characterization, they were obtained in a different way and gave us inspiration to deal with the high dimensional case.

We see that the topological condition (the existence of an incompressible torus in the boundary) that we consider here is specifically for dimension 3. For high dimension $3\leq n\leq7$, the situation is quite different, the problem is much more delicate and much more involved. However, extending to compact manifolds with boundary some of the ideas developed by Schoen-Yau \cite{SCHYau}, such as defining a class of manifolds via homology groups and using a descendent argument to recover the 3-dimensional case, we were able to obtain a type of classification result for high dimension.

\begin{theorem}\label{main2} 
Let $(M, \partial M)$ be a $(n+2)$-dimensional manifold, , $3\leq n+2 \leq 7$, such that there is a non-zero degree map $F:M\rightarrow \Sigma\times T^n$ such that $F(\partial M)=\partial \Sigma\times T^n$ , where $(\Sigma,\partial \Sigma)$ is a connected surface which is not a disk. Then there exists no metric on $M$ with positive scalar curvature and mean convex boundary. However, if $\Sigma$ is not a disk or a cylinder, then there exists no metric on $M$ with non-negative scalar curvature and mean convex boundary.
\end{theorem}

 As a consequence of the result above, we conclude that if $N$ is a closed $n$-dimensional manifold, then  $(T^{n-2}\times \mathring{T}^2 )\# N$ does not admit a metric of non-negative scalar curvature and mean convex boundary and $(I\times T^{n-1})\#N$ does not admit a metric of positive scalar curvature and mean convex boundary.

This article is organised as follows. In Section 2, we present some auxiliaries results
to be used in the proof of the main results. In Section 3, we present the discussion on the 3-dimensional case and the proof of the Theorem \ref{main1}. In section 4, we discuss the problem for the high dimensional case and present the proof of the Theorem \ref{main2}.

\subsection*{Acknowledgments} 
The first author was partially supported by CNPq-Brazil (Grant 312598/2018-1). The second author was partially supported by CAPES-Brazil (Grant 88882.184181/2018-01) and CNPq-Brazil (Grant 141904/2018-6).



\section{Preliminaries and Technical Results}
Let $(M,\partial M,g)$ be a Riemannian manifold of dimension $n\geq 3$. Assume that $R_{g},H_{g}\geq 0$ and $Vol_{g}(M)=1$, where $H_g$ denote the mean curvature of $\partial M$ with respect to the outward unit normal vector. For each Riemannian metric $\tilde{g}$ on $M$ consider $\lambda(\tilde{g})\in \mathbb{R}$ and $\Phi_{\tilde{g}}\in C^{\infty}(M)$ satisfying:

$$
\left\{
\begin{array}{ccccccc}
-\Delta_{\tilde{g}}\Phi_{\tilde{g}}+c_nR_{\tilde{g}}\Phi_{\tilde{g}} &=& \lambda({\tilde{g}})\Phi_{\tilde{g}}  \\
\displaystyle\frac{\partial \Phi_{\tilde{g}}}{\partial \nu_{\tilde{g}}} &=& -2c_nH_{\tilde{g}}\Phi_{\tilde{g}} \\
\displaystyle\int_M\Phi_{\tilde{g}}dv_{\tilde{g}} &=& 1 \\
\end{array}
\right.
$$
\noindent where $\nu_{\tilde{g}}$ denote the outward unit normal vector of the boundary $\partial M$ in $(M,\tilde{g})$ and $c_n:=\frac{(n-2)}{4(n-1)}$. Note that, as we are considering, we can assume that $\Phi_{\tilde{g}}>0$.

Moreover, note that
\begin{eqnarray*}
\lambda(\tilde{g})&=& -\int_M\Delta_{\tilde{g}}\Phi_{\tilde{g}}dv_{\tilde{g}}+c_n\int_M R_{\tilde{g}}\Phi_{\tilde{g}} dv_{\tilde{g}}\\
&=&-\int_{\partial M}\frac{\partial \Phi_{\tilde{g}}}{\partial \nu_{\tilde{g}}}d\sigma_{\tilde{g}}+c_n\int_M R_{\tilde{g}}\Phi_g dv_{\tilde{g}}.
\end{eqnarray*}

Therefore,
$$\lambda(\tilde{g})= 2c_n\int_{\partial M}\Phi_{\tilde{g}}H_{\tilde{g}}d\sigma_{\tilde{g}} +c_n\int_M R_{\tilde{g}}\Phi_{\tilde{g}} dv_{\tilde{g}}.$$

\begin{lemma}
Let $(M,\partial M,g)$ be a Riemannian $n$-dimensional manifold, $n\geq 3$,such that  $R_{g},H_{g}\geq 0$ and $Vol_{g}(M)=1$. If $\lambda(g)=0$ then 

$$D\lambda_{g}(h)= -c_n\int_{\partial M}\langle h,B_g\rangle d\sigma_g -c_n\int_M\langle h, Ric_g \rangle dv_g,$$\noindent for every symmetric tensor $h\in \mathcal{T}^{(2,0)}(M)$, where $B_g$ and $Ric_g$ is the second fundamental form of $\partial M$ in $(M,g)$ and the Ricci curvature of $(M,g)$, respectively.
\end{lemma}

\begin{proof}
Firstly, note that $\lambda(g)=0$ implies that $ R_{g}\equiv 0, \ H_{g}\equiv 0$ and $\Phi_{g}\equiv 1$. Let $h\in \mathcal{T}^{(2,0)}(M)$ be a symmetric tensor. Consider $g(t)$ for each $t\in (-\epsilon,\epsilon)$ a smooth family of Riemannian metrics on $M$ in a such way that $g(0)=g$ e $g'(0)=h$. Denote by
\[
\lambda(t):=\lambda(g(t)), \ h(t):=g'(t), \ R(t):=R_{g(t)}\quad\mbox{and}\quad \ H(t):=H_{g(t)}\,.
\]

As $ R_{g}\equiv 0, \ H_{g}\equiv 0$ and $\Phi_{g}\equiv 1$, we obtain that  
$$D\lambda_{g}(h)=\lambda'(0)= 2c_n\int_{\partial M}H'(0) d\sigma_g +c_n\int_M R'(0) dv_g.$$

We have that
\[
R'(t)=-\langle h(t), Ric_{g(t)} \rangle + div_{g(t)}(div_{g(t)}(h(t))-d(tr_{g(t)} \ h(t))).
\]

Hence,  

\begin{eqnarray*}
D\lambda_{g}(h) & = & 2c_n\int_{\partial M}H'(0) d\sigma_{g} -c_n\int_M\langle h, Ric_{g} \rangle dv_{g} \\
&+& c_n\int_M div_{g}(div_{g}(h)-d(tr_{g} \ h))dv_{g}\\
& = & 2c_n\int_{\partial M}H'(0) d\sigma_{g} -c_n\int_M\langle h, Ric_{g} \rangle dv_{g} \\
&+&c_n\int_{\partial M}\langle (div_{g}(h))^{\#}-(d(tr_{g} \ h))^{\#},\nu\rangle d\sigma_{g}\\
& = & c_n\int_{\partial M}(2H'(0)+X) d\sigma_{g} -c_n\int_M\langle h, Ric_{g}\rangle dv_{g},
\end{eqnarray*}
\noindent where $\nu:=\nu_g$ and $X:=\langle (div_{g}(h))^{\#}-(d(tr_{g} \ h))^{\#},\nu\rangle $.

\

{\bf Einstein convention and notation:}

\noindent{\bf (i)} Without a summation symbol, lower and upper index indicate a summation from $1$ to $n-1$.\\
{\bf (ii)} $\nabla^t$ denote the Riemannian connection  of $(M,g(t))$, $\nabla:=\nabla^0$.\\
{\bf (iii)}$B_t$ denote the second fundamental form of  $\partial M$ in $(M,g(t)).$\\

Consider $(x_1,\cdots,x_n)$ a local chart on $M$ such that $(x_1,\cdots,x_{n-1})$ is a local chart on $\partial M$ and $\partial_n=\nu$. We divide the proof in some steps.

\begin{enumerate}
\item[{\bf Step 1:}] Computation of $X$ in $\partial M$. 
\end{enumerate}
We have that
\[
d(tr_{g} \ h)=\sum_{k=1}^n\partial_k\left(\sum_{i,j=1}^ng^{ij}h_{ij}\right)dx^k \ \text{and} \ div_{g}(h)=\sum_{k=1}^n (div_{g}(h))_k dx^k\,.
\]
It follows that
\[
(d(tr_{g} \ h))^{\#}= \sum_{i,j,k,l=1}^ng^{lk}\partial_k(g^{ij}h_{ij})\partial_l,
\]
and
\[
 (div_{g}(h))^{\#}= \sum_{k,l=1}^ng^{lk}(div_{g}(h))_k\partial_l=\sum_{i,j,k,l=1}^ng^{lk}g^{ij}(\nabla_ih)_{jk}\partial_l.
 \]
Thus,
\[
(div_{g}(h))^{\#}-(d(tr_{g} \ h))^{\#}=\sum_{l=1}^n\left\{\sum_{i,j,k,=1}^n\left(g^{lk}g^{ij}(\nabla_ih)_{jk}-g^{lk}\partial_k(g^{ij}h_{ij})\right)\right\} \partial_l.
\]
In $\partial M$, we get that $g_{nn}=g^{nn}=1$ and $g_{ln}=g^{ln}=0$, for every $l=1,\cdots n-1$. Hence,  
\begin{eqnarray*}
X & = & \sum_{i,j=1}^n\left(g^{ij}(\nabla_ih)_{jn}-\nu(g^{ij}h_{ij})\right)\\
& = & g^{ij}(\nabla_ih)_{jn}+\nu(h_{nn})-\nu(g^{ij}h_{ij})-\nu(h_{nn})\\
& = & g^{ij}(\nabla_ih)_{jn}-\nu(g^{ij})h_{ij}-g^{ij}\nu(h_{ij})\\
& = & g^{ij}(\nabla_ih)_{jn}+g^{ik}g^{jl}\nu(g_{kl})h_{ij}-g^{ij}\nu(h_{ij})\\
& = & g^{ij}(\nabla_ih)_{jn}+2g^{ik}g^{jl}(B_g)_{kl}(h)_{ij}-g^{ij}\nu(h_{ij})\\
& = & g^{ij}(\nabla_ih)_{jn}+2\langle h,B_g\rangle -g^{ij}\nu(h_{ij})
\end{eqnarray*}

\begin{enumerate}
\item[{\bf Step 2:}]  Computation of $H'(0)$:
\end{enumerate}
We have $H(t)=g^{ij}_t(B_t)_{ij}$. Hence,
\begin{eqnarray*}
H'(t)&=&\frac{d}{dt}(g^{ij}_t)(B_t)_{ij}+g^{ij}_t\frac{d}{dt}(B_t)_{ij}\\
&=&-g^{ik}_tg^{jl}_t(h_t)_{kl}(B_t)_{ij}+ tr \ \left(\frac{d}{dt}B_t\right)\\
&=&-\langle h(t),B(t)\rangle+ tr \ \left(\frac{d}{dt}B_t\right).
\end{eqnarray*}
Lets focus our attention on $tr \ \left(\displaystyle\left.\frac{d}{dt}\right|_{t=0}B_t\right)$. We have that

$$(B_t)_{ij}=-g_t(\nu_{g_t},\nabla_i^t\partial_j).$$
$$\Rightarrow \left.\frac{d}{dt}\right|_{t=0}(B_t)_{ij}=-h(\nu,\nabla_i\partial_j)-\left\langle \left.\frac{d}{dt}\right|_{t=0}(\nu_{g_t}),\nabla_i\partial_j\right\rangle-\left\langle \nu, \left.\frac{d}{dt}\right|_{t=0}(\nabla_i^t\partial_j)\right\rangle,$$ \noindent where $g=\langle .,. \rangle.$

\begin{claim}\label{a1} For every $X,Y,Z\in \mathcal{X}(M)$, we obtain that
$$2\left\langle \left.\frac{d}{dt}\right|_{t=0}(\nabla^t_XY),Z\right\rangle=(\nabla_Xh)(Y,Z)+(\nabla_Yh)(X,Z)-(\nabla_Zh)(X,Y).$$
\end{claim}

It follows from the claim \ref{a1} that

\begin{eqnarray*}
\left.\frac{d}{dt}\right|_{t=0}(B_t)_{ij} & = & -h(\nu,\nabla_i\partial_j)-\left\langle \left.\frac{d}{dt}\right|_{t=0}(\nu_{g_t}),\nabla_i\partial_j\right\rangle-\frac{1}{2}(\nabla_ih)_{jn}\\
& & -\frac{1}{2}(\nabla_jh)_{in}+\frac{1}{2}(\nabla_{\nu}h)_{ij}.
\end{eqnarray*}

\begin{claim}\label{a2} In $\partial M$,

$$\left.\frac{d}{dt}\right|_{t=0}(\nu_{g_t})=-g^{kl}h_{nk}\partial_l-\frac{1}{2}h_{nn}\nu.$$
\end{claim}

\begin{proof}
In $\partial M$, we have $(g_t)_{nk}=0$ and $(g_t)_{nn}=1$, for all $k=1,\cdots n-1$ and $t\in (-\epsilon,\epsilon)$. Thus,
$$0=\left.\frac{d}{dt}\right|_{t=0}(g_t)_{nk}=h_{nk}+\left\langle \left.\frac{d}{dt}\right|_{t=0}(\nu_{g_t}),\partial_k\right\rangle,$$
\noindent and 
$$0=\left.\frac{d}{dt}\right|_{t=0}(g_t)_{nn}=h_{nn}+2\left\langle \left.\frac{d}{dt}\right|_{t=0}(\nu_{g_t}),\nu\right\rangle.$$

Moreover, for all $k=1,\cdots n-1$, we have that
$$\left\langle \left.\frac{d}{dt}\right|_{t=0}(\nu_{g_t}),\nu\right\rangle=-\frac{1}{2}h_{nn},$$
\noindent and 
$$\left\langle \left.\frac{d}{dt}\right|_{t=0}(\nu_{g_t}),\partial_k\right\rangle=-h_{nk}.$$

Denote by 
$$\left.\frac{d}{dt}\right|_{t=0}(\nu_{g_t})=\sum_{l=1}^n a_l\partial_l.$$

Note that $$a_n= \left\langle \left.\frac{d}{dt}\right|_{t=0}(\nu_{g_t}),\nu\right\rangle=-\frac{1}{2}h_{nn}.$$

Also,
$$-h_{nk}=\left\langle \left.\frac{d}{dt}\right|_{t=0}(\nu_{g_t}),\partial_k\right\rangle=\sum_{i=1}^{n-1}a_ig_{ki}, \  \forall k=1,\cdots,n-1.$$

It follows that $$a_l=-g^{lk}h_{nk},  \  \forall l=1,\cdots,n-1.$$

Hence,
$$\left.\frac{d}{dt}\right|_{t=0}(\nu_{g_t})=\sum_{l=1}^{n-1} a_l\partial_l+a_n\nu=-g^{lk}h_{nk}\partial_l-\frac{1}{2}h_{nn}\nu.$$
\end{proof}

It follows from the claim \ref{a2} that 
\begin{eqnarray*}
\left\langle\left.\frac{d}{dt}\right|_{t=0}(\nu_{g_t}),\nabla_i\partial_j\right\rangle & = & -g^{lk}h_{nk}\langle \nabla_i\partial_j,\partial_l\rangle-\frac{1}{2}h_{nn}\langle\nabla_i\partial_j,\nu\rangle\\
& = & - g^{lk}h_{nk}\Gamma_{ij}^mg_{ml}+\frac{1}{2}h_{nn}(B_g)_{ij}\\
& = & - h_{nk}\Gamma_{ij}^k+\frac{1}{2}h_{nn}(B_g)_{ij}\\
\end{eqnarray*}

However,

$$-h(\nabla_i\partial_j,\nu)=-h_{nk}\Gamma_{ij}^k-h_{nn}\Gamma_{ij}^n=(B_g)_{ij}h_{nn}-h_{nk}\Gamma_{ij}^k,$$ \noindent since

$$\Gamma_{ij}^n=\frac{1}{2}\sum_{k=1}^ng^{nk}\{\partial_ig_{jk}+\partial_jg_{ik}-\partial_kg_{ij}\}=-\frac{1}{2}\nu(g_{ij})=-(B_g)_{ij}.$$

It implies that 
\begin{eqnarray*}
\left\langle\left.\frac{d}{dt}\right|_{t=0}(\nu_{g_t}),\nabla_i\partial_j\right\rangle & = & -h(\nabla_i\partial_j,\nu)-(B_g)_{ij}h_{nn}+\frac{1}{2}h_{nn}(B_g)_{ij}\\
& = & -h(\nabla_i\partial_j,\nu)-\frac{1}{2}h_{nn}(B_g)_{ij}.
\end{eqnarray*}

Hence,
$$\left.\frac{d}{dt}\right|_{t=0}(B_t)_{ij}=-\frac{1}{2}(\nabla_ih)_{jn}-\frac{1}{2}(\nabla_jh)_{in}+\frac{1}{2}(\nabla_{\nu}h)_{ij}+\frac{1}{2}h_{nn}(B_g)_{ij}.$$

Consequently,
$$tr \ \left(\left.\frac{d}{dt}\right|_{t=0}(B_t)\right)=-g^{ij}(\nabla_ih)_{jn}+\frac{1}{2}g^{ij}(\nabla_{\nu}h)_{ij}+\frac{1}{2}h_{nn}H_g.$$

As $H_g=0$, we obtain that

\begin{eqnarray*}
2H'(0) & = & -2\langle h,B_g\rangle+ 2tr \ \left(\left.\frac{d}{dt}\right|_{t=0}(B_t)\right)\\
& = & -2\langle h,B_g\rangle-2g^{ij}(\nabla_ih)_{jn}+g^{ij}(\nabla_{\nu}h)_{ij}\\
& = & -2\langle h,B_g\rangle-2g^{ij}(\nabla_ih)_{jn}+g^{ij}\nu(h_{ij})-2g^{ij}h(\nabla_i\nu,\partial_j).
\end{eqnarray*}

\begin{claim}\label{a3} In $\partial M$,
$$g^{ij}h(\nabla_i\nu,\partial_j)=\langle h,B_g\rangle.$$
\end{claim}
\begin{proof}
Denote by
$$\nabla_i\nu=\sum_{k=1}^n\Gamma_{in}^k\partial_k.$$

Note that, in $\partial M$, we have
$$ \Gamma_{in}^n=0 \ \text{e} \ \Gamma_{in}^k=g^{mk}(B_g)_{im}, \ \forall k=1,\cdots,n-1.$$

This implies that 
$$\nabla_i\nu=g^{mk}(B_g)_{im}\partial_k$$

Then, in $\partial M$, we obtain that
$$g^{ij}h(\nabla_i\nu,\partial_j)=g^{ij}g^{mk}(B_g)_{im}h_{kj}=\langle h,B_g\rangle.$$
\end{proof}

It follows from the claim \ref{a3} that
$$2H'(0)= -4\langle h,B_g\rangle-2g^{ij}(\nabla_ih)_{jn}+g^{ij}\nu(h_{ij}).$$

Therefore,
\begin{equation}\label{e11}
2H'(0)+\left.X\right|_{\partial M}= -2\langle h,B_g\rangle-g^{ij}(\nabla_ih)_{jn}.
\end{equation}

\begin{claim}\label{a4} In $\partial M$, we have
$$g^{ij}(\nabla_ih)_{jn}= -\langle h,B_g\rangle+ div_g^{\partial M}(\omega),$$\noindent
for some $\omega\in \Omega^{1}(\partial M).$
\end{claim}

\begin{proof}
It follows from the claim \ref{a3} that, in $\partial M$,
\begin{eqnarray*}
g^{ij}(\nabla_ih)_{jn} & = & g^{ij}\partial_i(h_{jn})-g^{ij}h(\nabla_i\partial_j,\nu)-g^{ij}h(\nabla_i\nu,\partial_j)\\
& = & g^{ij}\partial_i(h_{jn})-g^{ij}h(\nabla_i\partial_j,\nu)-\langle h,B_g\rangle.
\end{eqnarray*}

For $1\leq i,j \leq n-1$, in $\partial M$, we can write
$$\nabla_i\partial_j= (B_g)_{ij}\nu+\overline{\nabla}_i\partial_j,$$ \noindent where $\overline{\nabla}$ is the Riemannian connection of $(\partial M,g).$

Hence, since $H_g\equiv 0$, we obtain that

$$g^{ij}h(\nabla_i\partial_j,\nu)=h_{nn}H_g+g^{ij}h(\overline{\nabla}_i\partial_j,\nu)= g^{ij}h(\overline{\nabla}_i\partial_j,\nu).$$

This implies that, in $\partial M$,

$$g^{ij}(\nabla_ih)_{jn}=g^{ij}\partial_i(h_{jn})-g^{ij}h(\overline{\nabla}_i\partial_j,\nu)-\langle h,B_g\rangle.
$$

Define $\omega\in\Omega^1(\partial M)$ as
$$\omega:=\displaystyle\left.h(.,\nu)\displaystyle\right|_{\partial M}.$$

Note that
\begin{eqnarray*}
div_g^{\partial M}(\omega) & = & g^{ij}(\overline{\nabla}_i\omega)_j=g^{ij}\partial_i(\omega_j)-g^{ij}\omega(\overline{\nabla}_i\partial_j)\\
& = & g^{ij}\partial_i(h_{jn})-g^{ij}h(\overline{\nabla}_i\partial_j,\nu).
\end{eqnarray*}

Then, in $\partial M$,
$$g^{ij}(\nabla_ih)_{jn}= -\langle h,B_g\rangle+ div_g^{\partial M}(\omega).$$
\end{proof}

It follows from equality (\ref{e11}) and claim \ref{a4} that
$$2H'(0)+\left.X\right|_{\partial M}= -\langle h,B_g\rangle-div_g^{\partial M}(\omega).$$

Hence,
$$D\lambda_g(h) = -c_n\int_{\partial M}\langle h,B_g\rangle d\sigma_g -c_n\int_M\langle h, Ric_g \rangle dv_g - c_n\int_{\partial M}div_g^{\partial M}(\omega) d\sigma_g.$$

We conclude, since $\partial M$ is a closed manifold, that
$$D\lambda_g(h) = -c_n\int_{\partial M}\langle h,B_g\rangle d\sigma_g -c_n\int_M\langle h, Ric_g \rangle dv_g.$$
\end{proof}

\begin{proposition}
Let $(M,\partial M,g)$ be a Riemannian manifold of dimension $n\geq 3$ such that $R_{g},H_{g}\geq 0$, $Vol_{g}(M)=1$ and $\lambda(g)=0$. The metric $g$ is a critical point of the functional $\lambda$ if and only if $(M,g)$ is Ricci flat with totally geodesic boundary. 
\end{proposition}

\begin{remark}\label{rr1}
Let $(M,\partial M,g)$ be a Riemannian manifold of dimension $n\geq 3$.  Define the following pair of operators acting in $C^{\infty}(M)$:
$$
\left\{
\begin{array}{ccccccc}
L_g & = & -\Delta_g\varphi+c_nR_g\varphi \ \text{in} \ M\\
T_g & = & \displaystyle\frac{\partial\varphi}{\partial \nu}+2c_nH_g^{\partial M}\varphi \ \text{on} \ \partial M \\
\end{array}
\right.
$$
\noindent where $\nu$ denotes the outward unit normal vector of the boundary $\partial M$ in $(M,g)$ and $c_n:=\frac{(n-2)}{4(n-1)}$. Consider the first eigenvalue $\lambda_1(M,g)$ of $L_g$ with boundary condition $T_g$:

\begin{equation}\label{eee}
\displaystyle\left\{
\begin{array}{ccccccc}
L_g(\varphi) &=& \lambda_1(M,g)\varphi \ \text{in} \ M\\
T_g(\varphi) &= & 0  \ \text{on} \ \partial M \\
\end{array}
\right.
\end{equation}

We have that,

$$\lambda_1(M,g)=\inf_{0\not\equiv \varphi\in H^1(M)} \displaystyle\frac{\displaystyle\int_M(|\nabla_g\varphi|^2+c_nR_g\varphi^2)dv_g+2c_n\displaystyle\int_{\partial M}H^{\partial M}_g\varphi^2 d\sigma_g}{\displaystyle\int_{M} \varphi^2dv_g}.$$

We can choose a positive function $\varphi\in C^{\infty}(M)$ solution of (\ref{eee}).  The conformal metric $h=\varphi^{\frac{4}{n-2}}g$ is such that
$$\left\{
\begin{array}{ccccccc}
R_h &=& \lambda_1(M,g)\varphi^{-\frac{4}{n-2}} \ \text{in} \ M \\
H_h^{\partial M} &\equiv &  0 \ \text{on} \ \partial M \\
\end{array}
\right.
$$

In particular, this implies that if $\lambda_1(M,g)>0$ then $R_h>0$ and $H_h^{\partial M}\equiv 0$.
\end{remark}

\begin{proposition}\label{prop31}
Let $(M,\partial M,g)$ be a Riemannian manifold of dimension $n\geq 3$ such that  $R_g\geq0$ and $H_g\geq 0$. Then $M$ admits a metric with positive scalar curvature and minimal boundary or $(M,g)$ is Ricci flat with totally geodesic boundary.
\end{proposition}

\begin{proof}
We can assume that $Vol_g(M)=1$. Hence, we obtain $\lambda(g)\geq 0$. If $\lambda(g)>0$, then there exists a metric on $M$ with positive scalar curvature and minimal boundary (See remark \ref{rr1}). 

Then, assume that $\lambda(g)=0$. If $D\lambda_g\equiv 0$, we have that $g$ is a critical point of the functional $\lambda$. Consequently,  $Ric_g\equiv 0$ and  $B_g\equiv 0$. If $D\lambda_g\not\equiv 0$, there exists a symmetric tensor $h_0\in\mathcal{T}^{(2,0)}(M)$ such that $D\lambda_g(h_0)>0$. Consider a family of metrics on $M$, $g(t)=g+th_0$, $t\in (-\epsilon,\epsilon).$ Since $\lambda'(0)=D\lambda_g(h_0)>0$, we obtain that there exists $\theta\in (0,\epsilon)$ such that the function $t\in (-\theta,\theta)\mapsto \lambda(t)\in \mathbb{R}$ is an increase function. Since $\lambda(g)=0$, we get that $\lambda(t)>0$ for all $t\in (0,\theta)$. Therefore, for each $t\in (0,\theta)$ there is a metric $\tilde{g}_t$ on $M$ such that $R_{\tilde{g}_t}>0$ and $H_{\tilde{g}_t}^{\partial M}\equiv 0$ (See remark \ref{rr1}).
\end{proof}

\begin{proposition}\label{p41}
Let $(M,\partial M,g)$ be a Riemannian manifold of dimension $n+1\geq 3$ such that $R_g>0$ and $H_g^{\partial M}\geq 0$. Then every free-boundary stable minimal hypersurface in $M$ has a metric with positive scalar curvature and minimal boundary. \end{proposition}

\begin{proof}
Consider $\Sigma$ a free-boundary stable minimal in $M$. It follows from the second variation formula for the volume that
\[
\int_{\Sigma}|\nabla \varphi|^2dv_g\geq \int_{\Sigma} \varphi^2(Ric_g(N,N)+|B_g^{\Sigma}|^2)dv_g+\int_{\partial\Sigma}\varphi^2B_g^{\partial M}(N,N)d\sigma_g
\]
for every $\varphi\in C^{\infty}(\Sigma)$, where $N$ denotes a unit vector field on $\Sigma$ in $(M,g)$. As $R_g>0$, it follows from the Gauss equation that
\[
Ric_g(N,N)+|B_g^{\Sigma}|^2=\frac{1}{2}(R_g-R_g^{\Sigma}+|B_g^{\Sigma}|^2)>-\frac{1}{2}R_g^{\Sigma}.
\]
Hence,
\[
\int_{\Sigma}|\nabla \varphi|^2dv_g> -\frac{1}{2}\int_{\Sigma} \varphi^2R_g^{\Sigma}dv_g+\int_{\partial\Sigma}\varphi^2B_g^{\partial M}(N,N)d\sigma_g\,,
\]
for every $\varphi\in C^{\infty}(\Sigma)$. Since $H_g^{\partial M}\geq 0$, and $\Sigma$ is a free-boundary hypersurface in $(M,g)$, we obtain 
\[
B_g^{\partial M}(N,N)=H_g^{\partial M}-H_g^{\partial \Sigma}\geq -H_g^{\partial \Sigma}.
\]
Thus, 
\[
\int_{\Sigma}|\nabla \varphi|^2dv_g> -\frac{1}{2}\int_{\Sigma} \varphi^2R_g^{\Sigma}dv_g-\int_{\partial\Sigma}\varphi^2H_g^{\partial \Sigma}d\sigma_g\,,
\]
for every $\varphi\in C^{\infty}(\Sigma)$. Consequently,
\[
\int_{\Sigma}|\nabla \varphi|^2dv_g+c_n\int_{\Sigma} \varphi^2R_g^{\Sigma}dv_g+2c_n\int_{\partial\Sigma}\varphi^2H_g^{\partial \Sigma}d\sigma_g>(1-2c_n)\int_{\Sigma}|\nabla \varphi|^2dv_g\,,
\]
for every $0\not\equiv\varphi\in H^1(\Sigma)$, where $c_n=\frac{n-2}{4(n-1)}$. It follows that
\[
\lambda=\inf_{0\not\equiv\varphi\in H^1(\Sigma)}\displaystyle\frac{\displaystyle\int_{\Sigma}|\nabla \varphi|^2dv_g+c_n\displaystyle\int_{\Sigma} \varphi^2R_g^{\Sigma}dv_g+2c_n\displaystyle\int_{\partial\Sigma}\varphi^2H_g^{\partial \Sigma}d\sigma_g}{\displaystyle\int_{\Sigma}\varphi^2dv_g}>0.
\]
Therefore, there exists a metric on $\Sigma$ with positive scalar curvature and minimal boundary (See remark \ref{rr1}).
\end{proof}

\section{3-dimensional case}

Let $M$ be a smooth 3-dimensional manifold. A sphere which is either properly embedded in $M$ or contained in $\partial M$ is said to be incompressible if it does not bound a ball in $M$. A connected surface $\Sigma$  which is either properly embedded in $M$ or contained in $\partial M$ and which is not a sphere, is said to be incompressible if the homomorphism $\pi_1(\Sigma)\hookrightarrow \pi_1(M)$ is injective. A connected embedded surface $(\Sigma,\partial \Sigma)\subset (M,\partial M)$ is said to be $\partial$-incompressible if the homomorphism $\pi_1(\Sigma,\partial\Sigma)\hookrightarrow \pi_1(M,\partial M)$ is injective. A properly embedded connected surface $(\Sigma,\partial \Sigma)\subset (M,\partial M)$ is said to be {\it essential} if it is incompressible and $\partial$-incompressible.

\begin{theorem}[Chen, Fraser e Pang, \cite{Fraser}]\label{t32}
Let $(M,\partial M, g)$ be a Riemannian 3-dimensional manifold. If $(\Sigma,\partial \Sigma)$ is a connected surface which is not a disk and $f:(\Sigma,\partial \Sigma)\rightarrow (M,\partial M)$ is a continuos map such that $$f_*:\pi_1(\Sigma)\rightarrow\pi_1(M) \ \ \text{e} \ \  f_*^{\partial}: \pi_1(\Sigma,\partial\Sigma)\rightarrow\pi_1(M,\partial M),$$\noindent are injectives, then there exists a free-boundary minimal immersion $F:(\Sigma,\partial \Sigma)\rightarrow (M,\partial M)$  and it minimizes area among the maps $h:(\Sigma,\partial \Sigma)\rightarrow (M,\partial M)$ such that $h_*$ and $h_*^{\partial}$ are injectives.
\end{theorem}

\begin{theorem}[Chen, Fraser e Pang, \cite{Fraser}]\label{t31}
Let $(M,\partial M, g)$ be a Riemannian 3-dimensional manifold such that $H_g\geq 0$. If $(\Sigma,\partial \Sigma)$ is a connected surface and $f:(\Sigma,\partial \Sigma)\rightarrow (M,\partial M)$ is a free-boundary, minimal and stable immersion, then  
\begin{enumerate}
\item[(1)] If $R_g>0$, we obtain that $\Sigma$ is a disk.
\item[(2)] If $R_g\geq 0$, we obtain that either $\Sigma$ is a disk or $(\Sigma,g)$ is a flat cylinder with totally geodesic boundary.
\end{enumerate}
\end{theorem}

Define $\tilde{\mathcal{C}}_3$ as the set of all smooth $3$-dimensional manifolds $(M,\partial M)$ such that there is no continuous map $f:(\Sigma,\partial \Sigma)\rightarrow (M,\partial M)$ with $f_*$ and $f_*^{\partial}$ are injectives, where $(\Sigma,\partial \Sigma)$ is a connected surface with genus $l$ and $k$ boundary components satisfying: $l=0$ and $k\geq 3$ or $l\geq 1$ and $k\geq 1$.
\begin{remark} As a consequence of the Theorems \ref{t31} e \ref{t32} we have that if a $3$-dimensional manifold $(M,\partial M)$ is such that $M\not\in\tilde{\mathcal{C}}_3$, then there is no metric on $M$ with non-negative scalar curvature and mean convex boundary. 
\end{remark}
Define $\overline{\mathcal{C}}_3$ as the set of all smooth $3$-manifolds $(M,\partial M)$ such that there is no continuous map $f:(\Sigma,\partial \Sigma)\rightarrow (M,\partial M)$ with $f_*$ and $f_*^{\partial}$ injectives, where $(\Sigma,\partial \Sigma)$ is a connected surface with genus $l$ and $k$ boundary components satisfying: $l=0$ and $k\geq 2$ or $l\geq 1$ and $k\geq 1$.

\begin{example} Consider the solid torus $M=\mathbb{S}^1\times \mathbb{D}^2$. Note that, $\pi_1(M,\partial M)=0$. It follows that,
$M\in\overline{\mathcal{C}}_3$. 
\end{example}

\begin{example}
Consider the 3-dimensional manifold $M=\mathbb{S}^1\times \Sigma$, where $(\Sigma,\partial \Sigma)$ is a connected surface which is not a disk. Note that $\Sigma$ is essential in $M$. Therefore, $M\not\in\overline{\mathcal{C}}_3$. 
\end{example}

\begin{example}
Consider the 3-dimensional manifold $M=I\times \mathbb{S}^2$. Since $M$ is simply connected, we have that $M\in\overline{\mathcal{C}}_3$.
\end{example} 

\begin{example}\label{ee2}
Consider the 3-dimensional manifold $M=I\times S$, where $S$ is a closed surface with positive genus. Let $\Sigma=I\times\gamma$, where $\gamma$ is a closed curve which represents a non-trivial class in $\pi_1(S)$ and bounds a "hole" in $S$. Note that $\Sigma$ is an essential cylinder which is properly embedded in $M$. Therefore, $M\not\in\overline{\mathcal{C}}_3$.
\end{example}

\begin{lemma}\label{lema1}
Let $(M,\partial M, g)$  be a connected Riemannian 3-dimensional manifold such that $g$ is flat with totally geodesic boundary. Then, $M$ is covered by $I\times T^2$. In particular, $M\not\in\overline{\mathcal{C}}_3$.
\end{lemma}
\begin{proof} It follows from the theorem $5$ in \cite{MSY} that either $M$ is diffeomorphic to a $3$-dimensional handlebody or $M$ is covered by $I\times T^2$. Since $(M,g)$ is flat with totally geodesic boundary, we have that $(\partial M, g)$ is a flat surface. Assume the $M$ is a $3$-dimensional handlebody. In this case, we have that $\partial M$ is connected. It follows from the Gauss-Bonnet theorem that $\partial M$ is a $2$-dimensional torus. This implies that $M=\mathbb{S}^1\times \mathbb{D}^2$ and $\partial M$ is a stable minimal flat torus in $(M,g)$. But, this is a contradiction (see Theorem $8$ in (\cite{MSY}). Therefore,  $M$ is covered by $I\times T^2$. Consider then $p:I\times T^2\rightarrow M$ a covering map and $C$ an essential cylinder which is properly embedded in $I\times T^2$. Define $f=p\circ i:(C,\partial C)\rightarrow (M,\partial M)$, where $i:C\rightarrow I\times T^2$ is the inclusion map. Note that $f_*$ and $f_*^{\partial}$ are injective. Therefore, $M\not\in\overline{\mathcal{C}}_3$.
\end{proof}
\begin{theorem}\label{pp11}
Let $(M,\partial M)$ be a smooth 3-dimensional manifold. Assume that the connected components of $\partial M$ are spheres or incompressible tori, but at least one of the components is a torus. Then $M\not\in\overline{\mathcal{C}}_3$. However, if the number of the incompressible tori in $\partial M$ is exactly one, then $M\not\in\tilde{\mathcal{C}}_3$.
\end{theorem}
\begin{proof} First, $M$ contains a  properly embedded, connected and incompressible surface $(\Sigma,\partial \Sigma)$ such that $0\not=[\partial\Sigma]\in H_1(\partial M)$ (see lemma 6.8 in \cite{Hempel}). 

\begin{Claim} $\Sigma$ is not a disk. 
\end{Claim}
In fact, assume that $\Sigma$ is a disk. As  $0\not=[\partial\Sigma]\in H_1(\partial M)$, we have that $\partial \Sigma$ represent a non-trivial class in $\pi_1(\partial M)$. Hence, $\partial\Sigma$ are in a connected component $T$ of $\partial M$ which is a torus (see the figure \ref{f2}).
\begin{figure}[H]
\centering
\includegraphics[scale=0.8]{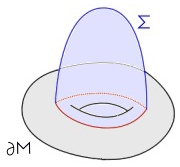}
\caption{$[\partial \Sigma]\in\pi_1(\partial M)$ is not trivial}
\label{f2}
\end{figure}

Note that the inclusion $i:\Sigma\rightarrow M$ represents a non-trivial class of $\pi_2(M,T)$. As $T$ is a   incompressible torus in $M$, we have that $\pi_1(T)\hookrightarrow \pi_1(M)$ is injective. Thus, we have the following exact sequence 
\[
\cdots\rightarrow \pi_2(T)\rightarrow \pi_2(M)\rightarrow \pi_2(M,T)\rightarrow 0.
\]

This implies that the map $\pi_2(M)\hookrightarrow \pi_2(M,T)$ is onto. Hence, there exists $f:\Sigma\rightarrow M$ with $f(\partial \Sigma)=x_0\in T$ and $0\not=[f]\in\pi_2(M)$ such that $[i]=[f]$ in $\pi_2(M,T)$. Consequently, $\partial \Sigma$ represents a non-trivial class in $\pi_1(T)$. But this is a contradiction. Therefore, $\Sigma$ is not a disk.

As $\Sigma$ is an incompressible surface, which is not a disk, we have that each connected component of $\partial\Sigma$ represents a non-trivial class in $\pi_1(\partial M)$. This implies that $\partial \Sigma$  is contained in the union of the connected components of $\partial M$ whose are torus. Hence, either $\Sigma$ is $\partial$-incompressible or it is a cylinder $\partial$-compressible (See lemma $1.10$ in \cite{Hacther}). 

\begin{Claim} $\Sigma$ is not a $\partial$-compressible cylinder.
\end{Claim}
In fact, suppose that $\Sigma$ is a $\partial$-compressible cylinder. In this case, we have that  connected components $\alpha_1$ and $\alpha_2$ of $\partial \Sigma$ are contained in a same torus of $\partial M$. Consequently, we have only two possible situation for the circles $\alpha_1$ and $\alpha_2$, as we can see in the figures below. 
\begin{figure}[H]
\centering
\includegraphics[scale=0.6]{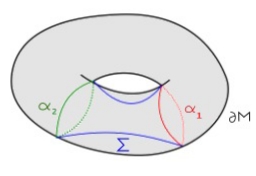}
\includegraphics[scale=0.3]{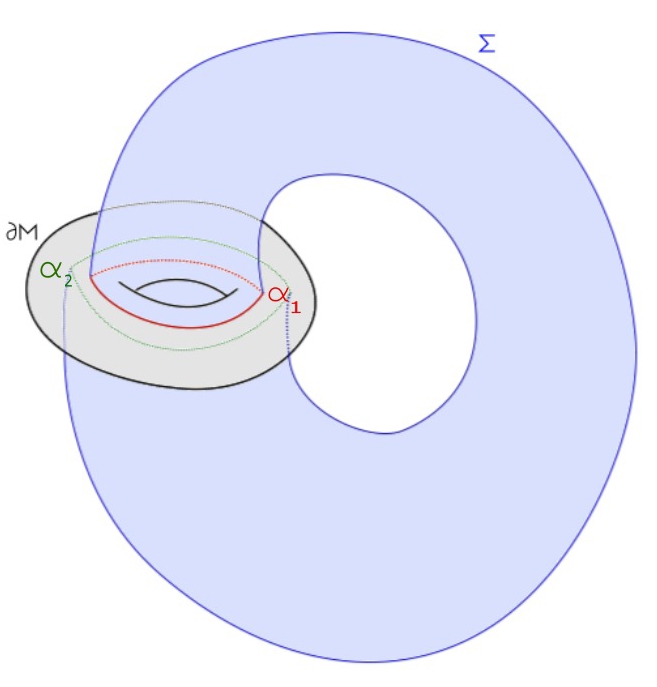}
\end{figure}
Note that in both situation we have that $\alpha_1$ and $\alpha_2$ are homologous in $\partial M$. This implies that $\partial\Sigma$ represent the trivial class in $H_1(\partial M)$. But this is a contradiction. Therefore, $\Sigma$ is not a $\partial$-compressible cylinder. 

Hence, $\Sigma$ is an essential surface in $M$ which is not a disk. Therefore, $M\not\in\overline{\mathcal{C}}_3$. However, note that if the number of the incompressible tori in $\partial M$ is exactly one, then the essential surface $\Sigma$ can not be a cylinder. This case, we have that $M\not\in\tilde{\mathcal{C}}_3$.
\end{proof}

\begin{remark}
The incompressibility condition of at least one torus of $\partial M$ in the proposition above is necessary. Actually, just consider the 3-dimensional manifold $M=\mathbb{S}^1\times \mathbb{D}^2.$ Note that the connected component of $\partial M$ is a compressible torus and $M\in\overline{\mathcal{C}}_3$.
\end{remark}

\begin{Cor} Let $(M,\partial M)$ be a smooth 3-dimensional manifold such that $\partial M$ is the disjoint union of exactly one torus and $k$ spheres, $k\geq 0$. If $M$ has a metric with non-negative scalar curvature and mean convex boundary then

$$M=N\#(\mathbb{S}^1\times \mathbb{D}^2)\#^k \mathbb{B}^3,$$
\noindent where $N$ is a closed 3-dimensional manifold.
\end{Cor}
\begin{proof}
The prime factorization of $M$ is

$$M=N_1\#\cdots\#N_s\# N'\#^k \mathbb{B}^3,$$
\noindent where $N_1,\cdots, N_s$ are closed and prime 3-dimensional manifolds and $N'$ is a prime 3-dimensional manifold  such that $\partial N'$ is a torus. If $M$ has a metric with non-negative scalar curvature and mean convex boundary, it follows from theorem \ref{pp11} that $\partial N'$ is a compressible torus in $N'$. Since $N'$ is prime, we have that $N'$ is irreducible. Consequently, $N'=\mathbb{S}^1\times \mathbb{D}^2$. Therefore,
$$M=N\#(\mathbb{S}^1\times \mathbb{D}^2)\#^k \mathbb{B}^3,$$
\noindent where $N=N_1\# \cdots\# N_s$.
\end{proof}

\begin{corollary}\label{corr1}
Let $(M_1,\partial M_1), \cdots, (M_k,\partial M_k)$  be 3-dimensional manifolds as in proposition \ref{pp11}, and $N_1,\cdots, N_s$ closed 3-dimensional manifolds. For every integer $l\geq 0$, we have that
\begin{enumerate}
\item[1.] $M_1\# \cdots\# M_k\#^l\mathbb{B}^3\not\in\overline{\mathcal{C}}_3$,
\item[2.] $M_1\# \cdots\# M_k\# N_1\# \cdots \# N_s\#^l\mathbb{B}^3\not\in\overline{\mathcal{C}}_3.$
\end{enumerate}
Moreover, if the number of the incompressible tori in $\partial M_1$ is exactly one then
\begin{enumerate}
\item[3.] $M_1\#^l\mathbb{B}^3\not\in\tilde{\mathcal{C}}_3,$
\item[4.] $M_1\# N_1\# \cdots \# N_s\#^l\mathbb{B}^3\not\in\tilde{\mathcal{C}}_3.$
\end{enumerate}
\end{corollary}

\begin{example}
Define the 3-dimensional manifolds $M_1=(\mathbb{S}^1\times \mathring{T}^2)\# N$ and $M_2=(\mathbb{S}^1\times \mathring{T}^2)\# (I\times \mathbb{S}^2)$, where $\mathring{T}^2$ is a torus minus an open disk and $N$ is a closed 3-dimensional manifold. It follows from the corollary \ref{corr1} that $M_1, M_2 \not\in\tilde{\mathcal{C}}_3$. This implies that $M_1$ and $M_2$ have no metric with non-negative scalar curvature and mean convex boundary.
\end{example}

\begin{proposition}\label{p32}
Let $(M,\partial M, g)$ be a 3-dimensional Riemannian manifold such that $R_g\geq 0$ and $H_g\geq 0$. Then either $M\in\overline{\mathcal{C}}_3$ or $(M,g)$ is flat with totally geodesic boundary. In particular, if $R_g>0$ and $H_g\geq 0$, then $M\in\overline{\mathcal{C}}_3$.
\end{proposition}

\begin{proof}
Note that as $R_g\geq 0$ and $H_g\geq 0$, we have that $M\in\tilde{\mathcal{C}}_3$. Assume that $M\not\in\overline{\mathcal{C}}_3$ and $g$ is not flat or $B_g^{\partial M}\not\equiv 0$. Since $M\in\tilde{\mathcal{C}}_3$, we have that there is a continuous map $f:(C,\partial C)\rightarrow (M,\partial M)$ such that $f_*$ and $f_*^{\partial}$ are injectives, where $C$ is a cylinder. As $g$ is not flat or $B_g^{\partial M}\not\equiv 0$, it follows from the proposition \ref{prop31} there exists a Riemannian metric $h$ on $M$ such that $R_h>0$ and $H_h^{\partial M}\equiv 0$. It follows from the Theorem \ref{t32} that there exists a stable free-boundary minimal immersion $F:(C,\partial C)\rightarrow (M,\partial M)$ with respect to the metric $h$. Hence, from Theorem \ref{t31}, we have a contradiction. This implies that $M\in\overline{\mathcal{C}}_3$ or $(M,g)$ is flat with totally geodesic boundary. It follows from the lemma \ref{lema1} that either $M\in\overline{\mathcal{C}}_3$ or $(M,g)$ is flat with totally geodesic boundary.
\end{proof}

\begin{example} Consider the $3$-dimensional manifold $I\times S$, where $S$ is a closed surface with positive genus. It follows from the proposition \ref{p32} that there is no metric on $I\times S$ with positive scalar curvature and mean convex boundary. In particular, there is no such metric on $I\times T^2$.
\end{example}

\begin{example}
It follows from the corollary \ref{corr1} that the 3-dimensional manifolds bellow are not in the set  $\overline{\mathcal{C}}_3$.

{\bf (1)}  $(I\times T^2)\# (I\times T^2)$

{\bf (2)}  $(\mathbb{S}^1\times \mathring{T}^2)\#(\mathbb{S}^1\times \mathring{T}^2)$

{\bf (3)}  $(I\times T^2)\# (S^1\times \mathring{T}^2)$

{\bf (4)} $(I\times T^2)\# (I\times \mathbb{S}^2)$

{\bf (5)} $(\mathbb{S}^1\times \mathring{T}^2)\#  (I\times \mathbb{S}^2)$

{\bf (6)} $(I\times T^2)\# N$, where $N$ is a closed 3-dimensional manifold.

Therefore, it follows from the proposition \ref{p32} that these manifolds have no metric with positive scalar curvature and mean convex boundary.

\end{example}

\section{$n$-dimensional case, $3\leq n\leq 7$}

In this section, we want to study possible generalisation of some results on the existence of certain metrics on $3$-dimensional manifolds to manifolds with dimension not greater than seven. The following theorem is a very important result from geometric measure theory which plays a fundamental role in our investigations.

\begin{theorem}[See Chapter $8$ in \cite{TGM2} and Theorem $5.4.15$ in \cite{TGM1}]\label{t41}
Let $(M,\partial M,g)$ be a Riemannian $n$-dimensional manifold, $3\leq n\leq 7$. Assume that $\alpha\in H_{n-1}(M,\partial M)$ is a non-trivial class. Then there exists a free-boundary, minimal and stable hypersurface $\Sigma$ properly embedded in $(M,g)$ which represents the class $\alpha$.
\end{theorem}

For $n\geq 4$, we define inductively the set $\tilde{\mathcal{C}}_n$ as  the set of all smooth $n$-dimensional manifolds  $(M,\partial M)$ such that every non-trivial class $\alpha\in H_{n-1}(M,\partial M)$ can be represented by a hypersurface $(\Sigma,\partial \Sigma)$ such that $\Sigma\in \tilde{\mathcal{C}}_{n-1}$.

\begin{theorem} \label{tt1}
Let $(M,\partial M)$ be a $n$-dimensional manifold such that $3\leq n\leq 7$ and $M\not\in\tilde{\mathcal{C}}_n$. Then there is no metric on $M$ with non-negative scalar curvature and mean convex boundary.
\end{theorem}

\begin{proof}
We note that it follows from a theorem above the it is true for $n=3$. We proof by induction on $n$. Assume the result is valid for $n-1$. Assume there exists a metric $g$ on $M$ such that $R_g\geq 0$ and $H_g^{\partial M}\geq 0$. It follows from proposition \ref{prop31} that two cases can occurs.

{\bf (1)} There exists a metric $h$ on $M$ such that $R_h>0$ and $H_h^{\partial M}\equiv 0$. In this case, Since $M\not\in \tilde{\mathcal{C}}_n$, from the theorem \ref{t41} we have there exists a compact, orientable, free-boundary, minimal and stable hypersurface $\Sigma$ properly embedded in $(M,h)$ such that $\Sigma\not\in \tilde{\mathcal{C}}_{n-1}$. From the proposition \ref{p41} there exists a metric on  $\Sigma$ with positive scalar curvature and minimal boundary. However, this is a contradiction since $\Sigma\not\in\tilde{\mathcal{C}}_{n-1}$ and from the induction hypothesis does not exists such metric.
 
{\bf (2)} Assume $Ric_g\equiv 0$ and $B_g^{\partial M}\equiv 0$. Arguing as in the item  $(1)$, we obtain there exists a compact orientable free-boundary stable minimal hypersurface $\Sigma$ in $(M,g)$ such that $\Sigma\not\in\tilde{\mathcal{C}}_{n-1}$. Since $\Sigma$ is free-boundary in $(M,g)$, we have 
\[
H_g^{\partial \Sigma}=H_g^{\partial M}-B_g^{\partial M}(N,N)\equiv 0,
\]
where $N$ is a unit vector field of $\Sigma$ em $(M,g)$. Also, it follows from the Gauss equation and of the stability of $\Sigma$ that $R_g^{\Sigma}\equiv 0$. However, this is a contradiction, since  $\Sigma\not\in\tilde{\mathcal{C}}_{n-1}$, from the  induction hypothesis, and does not exists a such metric on $\Sigma$ with null scalar curvature and minimal boundary.

\end{proof}

\begin{example}\label{ee1}
Consider the $n$-dimensional manifold $M^n=T^{n-2}\times \Sigma$, where $(\Sigma,\partial \Sigma)$ is a connected surface which is not a disk or a cylinder. We have the following chain of hypersurfaces 
\[
S^1\times \Sigma\subset T^{2}\times \Sigma\subset \cdots \subset T^{n-3}\times \Sigma\subset M.
\]
Since the surface $\Sigma$ is essential in $S^1\times \Sigma$, we obtain that $S^1\times \Sigma\not\in\tilde{\mathcal{C}}_3$.  Hence $M\not\in \tilde{\mathcal{C}}_n$, for every $n\geq 3$. It follows from theorem \ref{tt1} that  there is no metric on $M$ with non-negative scalar curvature and mean convex boundary, if $3\leq n\leq 7$.
\end{example}

For $n\geq 4$, we define inductively $\overline{\mathcal{C}}_n$ as  the set of all smooth $n$-dimensional manifolds $(M,\partial M)$ such that every non-trivial class $\alpha\in H_{n-1}(M,\partial M)$ can be represented by a hypersurface $(\Sigma,\partial \Sigma)$ such that $\Sigma\in \overline{\mathcal{C}}_{n-1}$.


\begin{proposition}\label{p42}
Let $(M,\partial M,g)$ be a Riemannian $n$-dimensional manifold, $3\leq n\leq 7$, such that $R_g\geq 0$ and $H_g\geq 0$. Then $M\in \overline{\mathcal{C}}_n$ or $(M,g)$ is Ricci flat with totally geodesic boundary. In particular, if $R_g>0$ and $H_g\geq 0$, then $M\in \overline{\mathcal{C}}_n$.
\end{proposition}

\begin{proof} 
It follows from the proposition \ref{p32} that the result is valid for $n=3$. Lets do it by induction on $n$. Assume the result is valid for $n-1$. Suppose that $Ric_g\not\equiv 0$ or $B_g\not\equiv 0$ and $M\not\in \overline{\mathcal{C}}_n$. It follows from the proposition \ref{prop31} that there exists a metric $h$ on $M$ such that $R_h>0$ and $H_h\equiv 0$. Since $M\not\in \overline{\mathcal{C}}_n$, from the theorem \ref{t41} we have there exists a free-boundary, minimal and stable hypersurface $\Sigma$ properly embedded in $(M,h)$ such that $\Sigma\not\in \overline{\mathcal{C}}_{n-1}$. From the  induction hypothesis we have that $\Sigma$ does not admit a metric with positive scalar curvature and minimal boundary. This is a contradiction with the proposition \ref{p41}. Therefore,  $M\in \overline{\mathcal{C}}_n$ or $(M,g)$ is Ricci flat with totally geodesic boundary.
\end{proof}

\begin{example} Consider the $n$-dimensional manifold $I\times T^{n-1}$ and the chain of hypersurfaces
\[ 
I\times T^2\subset I\times T^3\subset \cdots \subset I\times T^{n-1}.
\]
Since $I\times T^2\not\in \overline{\mathcal{C}}_3$, we have that $I\times T^{n-1}\not\in \overline{\mathcal{C}}_n$ for all $n\geq 3$. Hence, from the proposition \ref{p42}, there exists no metric on $I\times T^{n-1}$ with positive scalar curvature and mean convex boundary, for $3\leq n\leq 7$. 
\end{example}



\begin{theorem} Let $(M, \partial M)$ be a $(n+2)$-dimensional manifold, , $3\leq n+2 \leq 7$, such that there is a non-zero degree map $F:M\rightarrow \Sigma\times T^n$ such that $F(\partial M)=\partial \Sigma\times T^n$ , where $(\Sigma,\partial \Sigma)$ is a connected surface which is not a disk. Then there exists no metric on $M$ with positive scalar curvature and mean convex boundary. However, if $\Sigma$ is not a disk or a cylinder, then there exists no metric on $M$ with non-negative scalar curvature and mean convex boundary.
\end{theorem}

\begin{proof} Firstly, let we construct a chain of properly embedded hypersurfaces

$$(\Sigma_2,\partial \Sigma_2)\subset (\Sigma_3,\partial \Sigma_3)\subset \cdots (\Sigma_{n+1},\partial \Sigma_{n+1})\subset (M,\partial M)$$ \noindent
such that, for every $k=2,\cdots, n+1$, we have
\begin{enumerate}
\item $0\not=[\Sigma_k]\in H_k(\Sigma_{k+1},\partial \Sigma_{k+1});$
\item The map $F_k:=\left.F\right|_{\Sigma_k}:\Sigma_k\rightarrow \Sigma\times T^{k-2}$ has non-zero degree and $F_k(\partial\Sigma_k)=\partial \Sigma\times T^{k-2}$;
\end{enumerate}

We are going to construct the desired chain with an induction argument. Let us denote $(\Sigma_{n+2},F_{n+2})=(M,F)$. Without loss of generality,we assume that $F$ is a smooth function.  Now we state how to obtain $(\Sigma_k,F_k)$ from $(\Sigma_{k+1},F_{k+1})$, for $k=n+1,n,\cdots,2$, where $\Sigma_{k+1}$ is a $(k+1)$-dimensional manifold with non-empty boundary and $F_{k+1}:\Sigma_{k+1}\rightarrow \Sigma \times T^{k-1}$ is a non-zero degree map such that $F_{k+1}(\partial\Sigma_{k+1})=\partial \Sigma\times T^{k-1}$ . For this, consider the projection $p_{k-1}:\Sigma\times T^{k-1}\rightarrow S^1$ given by $p_{k-1}(x,(t_1,\cdots,t_{k-1}))=t_{k-1}$, for every $x\in \Sigma$ and $(t_1,\cdots,t_{k-1})\in T^{k-1}=\mathbb{S}^1\times\cdots \times \mathbb{S}^1$. Define $f_{k+1}=p_{k-1}\circ F_{k+1}.$ It follows from the Sard's Theorem that there is $\theta_{k-1}\in S^1$ which is a regular value of $f_{k+1}$ and $\partial f_{k+1}$. Define
$$\Sigma_{k}:=f_{k+1}^{-1}(\theta)=F_{k+1}^{-1}(\Sigma\times T^{k-2}\times\{\theta_{k-1}\}).$$

We have that $\Sigma_{k}$ is a hypersurface of $\Sigma_{k+1}$ such that 

$$\partial \Sigma_{k}=\Sigma_{k}\cap \partial \Sigma_{k+1}=F_{k+1}^{-1}(\partial\Sigma\times T^{k-2}\times\{\theta_{k-1}\})\cap \partial \Sigma_{k+1}\not=\emptyset.$$

Note that $\Sigma_{k}$ represent a non-trivial class in $H_{k}(\Sigma_{k+1},\partial \Sigma_{k+1})$ and $F_{k}:=\left.F_{k+1}\right|_{\Sigma_{k}}:\Sigma_k\rightarrow \Sigma \times T^{k-2}$ is a non-zero degree map with $F_k(\partial\Sigma_k)=\partial \Sigma\times T^{k-2}$ .

\begin{claim}\label{AF5} $\Sigma_3\not\in \overline{\mathcal{C}_3}$.
\end{claim}
In fact, since $0\not=\alpha=[\Sigma_2]\in H_2(\Sigma_{3},\partial \Sigma_{3})$, we have that there is a properly embedded surface $S\subset \Sigma_3$ which represents the homology class $\alpha$ such that its connected components are either essential surfaces, disks, or spheres (See \cite{BM}). Since $0\not=[S]=[\Sigma_2]\in H_2(\Sigma_{3},\partial \Sigma_{3})$ we have that $F_3(S)=\Sigma$ and the map $\left.F_3\right|_{S}:S\rightarrow\Sigma$ has non-zero degree, because

$$deg (\left.F_3\right|_{S})= deg (\left.F_3\right|_{\Sigma_2})= deg (F_2)\not = 0.$$

It follows that there is a connected component $(S',\partial S')$ of $S$ such that $\left.F_3\right|_{S'}: S'\rightarrow \Sigma$ has non-zero degree. This implies that $\chi(S')\leq \chi(\Sigma).$ Since $\Sigma$ is not a disk, we have that $\chi(S')\leq 0$. This implies that $S'$ is not a disk or a sphere. It follows that $S'$ is an essential surface in $\Sigma_3$ which is not a disk. Hence, $\Sigma_3\not\in \overline{\mathcal{C}_3}$.

\begin{claim}\label{AF6} $M\not\in \overline{\mathcal{C}}_{n+2}$.
\end{claim}
Consider a hypersurface $(S_3,\partial S_3)\subset (\Sigma_4,\partial\Sigma_4)$ such that $[S_3]=[\Sigma_3]\in H_3(\Sigma_{4},\partial \Sigma_{4})$. Note that $F_4(S_3)=\Sigma\times S^1$, $F_4(\partial S_3)=\partial\Sigma\times S^1$ and

$$deg (\left.F_4\right|_{S_3})=deg (\left.F_4\right|_{\Sigma_3})=deg (F_3)\not=0.$$

As in the proof of the claim \ref{AF5}, we can conclude that $S_3\not\in \overline{\mathcal{C}_3}$. This implies that $\Sigma_4\not\in \overline{\mathcal{C}_4}$. Inductively, we can show that $\Sigma_{n+1}\not\in \overline{\mathcal{C}}_{n+1}$. Hence, $M\not\in \overline{\mathcal{C}}_{n+2}$. 

\

It follows from proposition \ref{p41} that  there exists no metric on $M$ with positive scalar curvature and mean convex boundary. However, note that if $\Sigma$ is not a disk or a cylinder, we can replace $\overline{\mathcal{C}}$ by $\tilde{\mathcal{C}}$ in the claims \ref{AF5} and \ref{AF6} and conclude that $M\not\in \tilde{\mathcal{C}}_{n+2}$. Consequently, from  theorem \ref{tt1}, we have that  there exists no metric on $M$ with non-negative scalar curvature and mean convex boundary.
\end{proof}

\begin{remark}The following statements are two important consequences of the theorem above:

\begin{enumerate}
\item The manifold $(I\times T^{n-1})\# N$ admits no metric with positive scalar curvature and mean convex boundary.
\item The manifold $(\Sigma \times T^{n-2})\# N$ admits no metric with non-negative scalar curvature and mean convex boundary
\end{enumerate}
\noindent where $N$ is a closed manifold of dimension $3\leq n\leq 7$ and $(\Sigma,\partial \Sigma)$ is a connected surface which is not a disk or a cylinder.
\end{remark}

\end{document}